\documentclass[12pt]{amsart}
\usepackage{amssymb}
\newtheorem{theorem}{Theorem}
\newtheorem{proposition}[theorem]{Proposition}

\newtheorem{lemma}[theorem]{Lemma}

\def\dist{\operatorname{dist}}

\def\grad{\operatorname{grad}}
\def\supp{\operatorname{supp}}
\def\DD{\Bbb D}
\def\CC{\Bbb C}
\def\BB{\Bbb B}
\def\RR{\Bbb R}
\def\PSH{\mathcal{PSH}}
\def\L{\mathcal{L}}
\def\O{\mathcal{O}}
\def\eps{\varepsilon}
\def\phi{\varphi}

\title{On a local characterization of pseudoconvex domains}

\author[N.~Nikolov, P.~Pflug, P.~J.~Thomas, W.~Zwonek]
{Nikolai Nikolov, Peter Pflug, Pascal J.~Thomas, Wlodzimierz Zwonek}

\address
{Institute of Mathematics and Informatics\\ Bulgarian Academy of
Sciences\\ Acad. G. Bonchev 8, 1113 Sofia,
Bulgaria}\email{nik@math.bas.bg}

\address{Carl von Ossietzky Universit\"at Oldenburg\\
Institut f\"ur Mathe\-ma\-tik\\ Postfach 2503\\ D-26111 Oldenburg,
Germany}\email{pflug@mathematik.uni-oldenburg.de}

\address{Institut de Math\'ematiques\\
Universit\'e Paul Sabatier, 118 Route de Narbonne, 31062 Toulouse
Cedex 9, France} \email{pthomas@cict.fr}

\address{Instytut Matematyki, Uniwersytet Jagiello\'nski, Reymonta 4,
30-059 Krak\'ow, Poland}\email{Wlodzimierz.Zwonek@im.uj.edu.pl}



\begin{document}

\begin{thanks}{This paper was written during the stay of the third
named author at the Institute of Mathematics and Informatics of the
Bulgarian Academy of Sciences supported by a CNRS grant and the stay
of the fourth named author at the Universit\"at Oldenburg supported
by a DFG grant No 436 POL 113/106/0-2 (July 2008).}
\end{thanks}

\begin{abstract} Pseudoconvexity of a domain in $\Bbb C^n$ is
described in terms of the existence of a locally defined
plurisubharmonic/holo\-morphic function near any boundary point that
is unbounded at the point.
\end{abstract}

\maketitle

\section{Introduction and results}

It is well-known that a domain $D\subset\Bbb C^n$ is pseudoconvex if
and only if any of the following conditions holds:
\smallskip

(i) there is a smooth strictly plurisubharmonic function $u$ on
$D$ with $\lim_{z\to\partial D}u(z)=\infty;$

(ii) for any $a\in\partial D$ there is a $u_a\in\PSH(D)$ with
$\lim_{z\to a}u_a(z)=\infty;$

(iii) there is an $f\in\mathcal O(D)$ such that for any
$a\in\partial D$ and any neighborhood $U_a$ of $a$ one has that
$\limsup_{G\ni z\to a}|f(z)|=\infty$ for any connected component
$G$ of $D\cap U_a$ with $a\in\partial G;$

(iv) for any $a\in\partial D$ there is a neighborhood $U_a$ of $a$
and an $f_a\in\mathcal O(D\cap U_a)$ such that for any
neighborhood $V_a\subset U_a$ of $a$ and any connected component
$G$ of $D\cap V_a$ with $a\in\partial G$ one has  $\limsup_{G\ni
z\to a}|f_a(z)|\\=\infty$ (see Corollary 4.1.26 in \cite{Hor}).

If $D$ is $C^1$-smooth, we may assume that $D\cap U_a$ is connected
in (iii) and (iv).
\smallskip

Our first aim is to see that in (i) in general 'lim' cannot be
weakened by 'limsup' even if $D$ is $C^1$-smooth.

\begin{theorem}\label{1} For any $\varepsilon\in(0,1)$ there is a
non-pseudoconvex bounded domain $D\subset\Bbb C^2$ with
$C^{1,1-\varepsilon}$-smooth boundary and a negative function
$u\in\PSH(D)$ with $\limsup_{z\to a}u(z)=0$ for any $a\in\partial
D.$

In particular, $v:=-\log(-u)\in\PSH(D)$ with $\limsup_{z\to
a}v(z)=\infty$ for any $a\in\partial D.$
\end{theorem}

If we do not require smoothness of $D,$ following the idea presented
in the proof, we may just take $D=\{z\in\Bbb
C^n:\min\{||z||,||z-a||\}<1\},$ $0<||a||<2,$ $n\ge 2.$

On the other, this cannot happen if $D$ is $C^2$-smooth.

\begin{proposition}\label{2} Let $D\subset\Bbb C^n$ be a $C^2$-smooth
domain with the following property: for any boundary point
$a\in\partial D$ there is a neighborhood $U_a$ of $a$ and a function
$u_a\in\PSH(D\cap U_a)$ such that $\limsup_{z\to a}u_a(z)=\infty.$
Then $D$ is pseudoconvex.
\end{proposition}

However, if we replace 'limsup' by 'lim', we may remove the
hypothesis about smoothness of the boundary.

\begin{proposition}\label{3} Let $D\subset\Bbb C^n$ be a domain
with the following property: for any boundary point $a\in\partial
D$ there is a neighborhood $U_a$ of $a$ and a function
$u_a\in\PSH(D\cap U_a)$ such that $\lim_{z\to a}u_a(z)=\infty.$
Then $D$ is pseudoconvex.
\end{proposition}

Note that the assumption in Proposition \ref{3} is formally weaker
that to assume that $D$ is locally pseudoconvex.

\vskip 0.3cm

\noindent{\it Remark.} The three propositions above have real
analogues replacing (non)pseudoconvex domains by (non)convex domains
and plurisubharmonic functions by convex functions (for the analogue
of Proposition 3 use e.g. Theorem 2.1.27 in \cite{Hor} which implies
that if $D$ is a nonconvex domain in $\Bbb R^n,$ then there exists a
segment $[a,b]$ such that $c=\frac{a+b}{2}\in\partial D$ but
$[a,b]\setminus\{c\}\subset D$). The details are left to the reader.
\smallskip

Recall now that a domain $D\subset\Bbb C^n$ is called {\it locally
weakly linearly convex} if for any boundary point $a\in\partial D$
there is a complex hyperplane $H_a$ through $a$ and a neighborhood
$U_a$ of $a$ such that $H_a\cap D\cap U_a=\varnothing.$ D. Jacquet
asked whether a locally weakly linearly convex domain is already
pseudoconvex (see \cite{Jac}, page 58). The answer to this
question is affirmative by Proposition \ref{3}. The next
proposition shows that such a domain has to be even
taut\footnote{This means that $\mathcal O(\Bbb D,D)$ is a normal
family, where $\Bbb D\subset\Bbb C$ is the open unit disc. Note
that any taut domain is pseudoconvex and any bounded pseudoconvex
domain with $C^1$-smooth boundary is taut.} if it is bounded.

\begin{proposition}\label{4}
Let $D\subset\Bbb C^n$ be a bounded domain with the following
property: for any boundary point $a\in\partial D$ there is a
neighborhood $U_a$ of $a$ and a function $f_a\in\mathcal O(D\cap
U_a)$ such that $\lim_{z\to a}|f_a(z)|=\infty.$ Then $D$ is taut.
\end{proposition}

Let $D\subset\Bbb C^n$ be a domain and let $K_D(z)$ denote the
Bergman kernel of the diagonal. It is well-known that $\log
K_D\in\PSH(D).$ Recall that

(v) if $D$ is bounded and pseudoconvex, and $\limsup_{z\to
a}K_D(z)=\infty$ for any $a\in\partial D,$ then $D$ is an
$L_h^2$-domain of holomorphy ($L_h^2(D):=L^2(D)\cap\mathcal O(D)$)
(see \cite{Pfl-Zwo}).

We show that the assumption of pseudoconvexity is essential.

\begin{proposition}\label{10} There is a non-pseudoconvex bounded domain
$D\subset\Bbb C^2$ such that $\limsup_{z\to a}K_D(z)=\infty$ for any
$a\in\partial D$.
\end{proposition}

Note that the domain $D$ with $u=\log K_D$ presents a similar
kind of example as that in Proposition \ref{1} (however, the
domain has weaker regularity properties).

The example given in Proposition \ref{10} is a domain with
non-schlicht envelope of holomorphy. This is not accidental as the
following result shows.

\begin{proposition}\label{11} Let $D\subset\Bbb C^n$ be a domain such that
$\limsup_{z\to a}K_D(z)=\infty$ for any $a\in\partial D$. Assume
that one of the following conditions is satisfied:

-- the envelope of holomorphy $\hat D$ of $D$ is a domain in
$\CC^n$,

-- for any $a\in\partial D$ and for any neighborhood $U_a$ of $a$
there is a neighborhood $V_a\subset U_a$ of $a$ such that $V_a\cap
D$ is connected (this is the case when e.g. $D$ is a $C^1$-smooth
domain).

Then $D$ is pseudoconvex.
\end{proposition}

\noindent{\it Remark.} Note that the domain in the example is not
fat. We do not know what will happen if $D$ is assumed to be fat.

\vskip 0.3cm

Making use of the reasoning in \cite{Irg2} we shall see how
Proposition \ref{10} implies that the domain from this proposition
admits a function $f\in L_h^2(D)$ satisfying the
property $\limsup_{z\to a}|f(z)|=\infty$ for any $a\in\partial D.$

\begin{theorem}\label{12} Let $D$ be the domain from Proposition
\ref{10}. Then there is a function $f\in L_h^2(D)$ such that
$\limsup_{z\to a}|f(z)|=\infty$ for any $a\in\partial D.$
\end{theorem}

\section{Proof of Proposition 1}

First, we shall prove two lemmas.

\begin{lemma}\label{5}
For any $\eps \in (0,1)$ and $C_1$, $C_2>0$, there exists an $F\in
\mathcal C^{1, 1-\eps}(\RR)$ such that:

{\rm (i)} $\supp F\subset [-1,+1]$, $0\le F(x)\le C_1$ for all $x
\in \RR$;

{\rm (ii)} there is a dense open set $\mathcal U \subset [-1,+1]$
such that $F''(x)$ exists and $F''(x) \le -C_2<0$ for all $x \in
\mathcal U$;

{\rm (iii)} $F$ vanishes on a Cantor subset of $[-1,+1].$
\end{lemma}

\begin{proof} An elementary construction yields an even non-negative
smooth function $b$ supported on $[-3/4,+3/4]$, decreasing on $[0,
3/4]$, such that $b(x) = 1-4x^2$ for $|x|\le 1/4$, $|b'(x)| \le
C_3$, $-8 \le b''(x) \le C_4$ for all $x\in\Bbb R,$ where
$C_3,C_4>0.$

For any $a,p >0$, we set $b_{a,p}(x):= ab(x/p)$, $x\in\RR$.

We shall construct two decreasing sequences of positive numbers
$(a_n)_{n\geq 0}$ and $(p_n)_{n\geq 0}$, and intervals $\{
I_{n,i}, J_{n,i}, n \ge 0, 1\le i \le 2^n\}$.

Set $I_{0,1}:= (-1,+1)$ and $J_{0,1}:=[-p_0/4,p_0/4]$, where
$p_0<1$. Then $I_{1,1}:= (-1,-p_0)$ and $I_{1,2}:= (p_0,1)$.

In general, if the intervals of the n-th "generation" $ I_{n,i}$
are known, we require
\begin{equation}
\label{condpn} p_{n} < \frac{| I_{n,i}|}{2},
\end{equation}
where $|J|$ denotes the length of an interval $J$. Denote by
$c_{n,i}$ the center of $ I_{n,i}$ and put $J_{n,i}:=
[c_{n,i}-p_{n}/4, c_{n,i}+p_{n}/4]$. Denote respectively by $I_{n+1,
2i-1}$ and $I_{n+1, 2i}$ the first and second component of $I_{n,i}
\setminus J_{n,i}$.

Now we write
$$
f_n (x) := \sum_{i=1}^{2^n} b_{a_n,p_n}(x-c_{n,i}),\ x\in\Bbb R,
\quad F_n:= \sum_{m=0}^n f_m.
$$
Note that the terms in the sum defining $f_n$ have disjoint supports
contained in $[c_{n,i}-3p_{n}/4, c_{n,i}+3p_{n}/4]\subset
I_{n,i},$ ($J_{n,i}$ does not contain  the support of the
corresponding term in $f_n;$ it is only a place, where that term
coincides with a quadratical polynomial) so that $|f'_n(x)| \le
C_3a_n/p_n$. The function $F=\lim_{n\to\infty} F_n$ will be of class
$\mathcal C^1$ if
\begin{equation}
\label{condc1} \sum_{n=0}^\infty\frac{a_n}{p_n} < \infty.
\end{equation}
Also, note that
$$
|F''_n(x)|\le |F''_{n-1}(x)| + C_4 \frac{a_n}{p_n^2}\mbox{, so } \sup
|F''_n| \le C_4 \sum_{m=1}^n \frac{a_m}{p_m^2}.
$$

From now on we choose
\begin{equation}
\label{geom} \frac{a_n}{p_n^2} = BA^n\mbox{, for some }A>1, B>0
\mbox{ to be determined.}
\end{equation}
We then have $\sup |F''_n| \le C_4 B A^{n+1}/(A-1)$.

All the successive terms $f_m, m>n,$ are supported on intervals of
the form $I_{m,j}$, thus vanish on the interval $J_{n,i}$, so on
those intervals $F$ is a smooth function and
$$F''=F''_n=F''_{n-1}-8\frac{a_n}{p_n^2}\le C_4\frac{BA^{n}}{A-1}-8 B A^n;$$
therefore, if we choose
\begin{equation}
\label{condA} A>1+\frac{C_4}4,
\end{equation}
we have $F''(x)\le -4B A^n$ for all $x\in J_{n,i}$, and $1\le i
\le 2^n$.

Set $\mathcal U := \bigcup_{n,i} J_{n,i}^\circ$.  We have seen that
$| I_{n+1,i}|<|I_{n,j}|/2$ (and those quantities do not depend on
$i$ or $j$), so that the complement of $\mathcal U$ has empty
interior. This proves claim (ii), by choosing $B=C_2/4$. The other
claims are clear from the form of the function $F$, once we provide
the sequences $(a_n)$ and $(p_n)$ satisfying \eqref{geom},
\eqref{condA}, \eqref{condc1}, and \eqref{condpn}.

Let $a_n:= a_0 \gamma^n$, $p_n=p_0\delta^n$. Then \eqref{geom} is
satisfied by construction and $a_0 = Bp_0^2$. Fix
$\delta,p_0\in(0,1/2).$  It follows that $p_{n}<| I_{n,i}|/4$ for
all $n$ (by an easy induction). Hence, \eqref{condpn} holds.

By our explicit form, \eqref{condA} means that $\gamma \delta^{-2}
> 1+\frac{C_4}4$, while \eqref{condc1} means $\gamma \delta^{-1} <
1$, so with $\delta^{-1} > 1+\frac{C_4}4$, it is easy to choose
$\gamma$. Finally $\|F\|_\infty \le a_0 (1-\gamma)^{-1} < C_1$ for
$a_0$ small enough, which can be achieved by decreasing $p_0$
further.

Given any $\varepsilon >0$, we can modify the choices of $\delta$
and $\gamma$ to obtain that $F' \in \Lambda_{1-\eps}$ (the H\"older
class of order $1-\eps$). Given any two points $x,y \in [-1,+1]$ and
any integer $n \ge 1$,
$$
|F'(x) - F'(y)| \le |x-y| \| F''_n \|_\infty + 2 \sum_{m\ge n} \|
f'_m\|_\infty
$$
$$\le C \left( (\gamma \delta^{-2})^n |x-y| +
(\gamma \delta^{-1})^n \right),
$$
where $C>0$ is a positive  constant depending on the parameters we
have chosen. Take $n$ such that $\delta  |x-y|  \le \delta^n \le
|x-y| $. Then
$$
\frac{|F'(x) - F'(y)|}{|x-y|^{1-\eps}} \le C'  (\gamma
\delta^{-2+\eps})^n,
$$
and it will be enough to choose $\delta$ and $\gamma$ so that
$\gamma \delta^{-2+\eps}\le 1$ and $\gamma \delta^{-2} >
1+\frac{C_4}4$, which can be achieved once we pick $\delta$ small
enough.  The rest of the parameters are then chosen as above.
\end{proof}

\noindent{\it Remark.} It is clear that $F$ cannot be of class
$\mathcal C^2(\Bbb R)$. We do not know if our argument can be pushed
to get $F\in \mathcal C^{1,1}(\Bbb R).$
\smallskip

\begin{lemma}\label{6}
For any $\eps\in(0,1)$ there exists a non-pseudoconvex bounded
$\mathcal C^{1, 1-\eps}$-smooth domain $D\subset \CC^2$ boundary
such that $\partial D$ contains a dense subset of points of strict
pseudoconvexity.
\end{lemma}

\begin{proof}
We start with the unit ball and cave it in somewhat at the North
Pole to get an open set of points of strict pseudoconcavity on the
boundary. Let $r_0 < 1/3$ and for $x \in [0,1)$,
$$
\psi_0 (x) = \min\{\log (1-x^2), x^2-r_0^2\}. \footnote{Note that
the graphs of both functions cut inside the interval
$(r_0/2,r_0).$ Indeed, $x^2-r_0^2>\log(1-x^2)$ for $x\ge r_0^2$
and $x^2-r_0^2<\log(1-x^2)$ for $x\le r_0^2/2.$}
$$
We take $\psi$ a $\mathcal C^\infty$ regularization of $\psi_0$
such that $\psi=\psi_0$ outside of $(r_0/2,r_0)$. Consider the
Hartogs domain
$$
D_0 := \left\{ (z,w) \in \CC^2 : |z|<1, \log |w| < \frac12 \psi
(|z|) \right\}.
$$
Notice that $D_0 \setminus \{ |z| \le r_0 \} = \BB_2 \setminus \{
|z| \le r_0 \}$, so that $\partial D$ is smooth near $|z|=1$.

Now define $\Phi(z)=\Phi (x+iy) = F(x/r_0) \chi  (y/r_0)$, where $F$
is the function obtained in Lemma \ref{5}, and $\chi$ is a smooth,
even cut-off function on $\RR$ such that $0\le \chi \le 1$,
$\mbox{supp}\, \chi \subset (-2,2)$, and $\chi\equiv 1$ on $[-1,1]$.
We define
$$
D:= \left\{ (z,w) \in \CC^2 : |z|<1, \log |w| < \frac12 \psi (|z|)
+ \Phi (z) \right\}.
$$
Recall that for a Hartogs domain $\{\log|w|<\phi(z), |z|<1\},$ if
$\phi$ is of class $\mathcal C^2$ at $z_0$, a boundary point
$(z_0,w_0)$ with $|z_0|<1$ is strictly pseudoconvex (respectively,
strictly pseudoconcave) if and only if $\Delta \phi (z_0) <0$
(respectively, $\Delta \phi (z_0)>0$). Choosing an appropriate
regularization (convolution  by a smooth positive kernel of small
enough support), we may get that:
\begin{itemize}
\item
$\Delta \psi (|z|)\le-4$ for $|z|\ge r_0$,
\item
$\Delta \psi (|z|)=4$ for $|z|\le r_0/2$, and is always $\le 4$.
\end{itemize}
We consider points $z_0=x+iy$. If $|x|>r_0$, $\Phi (z_0)=0$ and we
have pseudoconvex points (the boundary is a portion of the
boundary of the ball).

On the other hand, when $x \in r_0 \mathcal U$ (where $\mathcal U$
is the dense open set defined in Lemma \ref{5}),
$$
\Delta \Phi (z_0) = \frac1{r_0^2} \Bigl( F''(x/r_0) \chi (y/r_0) +
F(x/r_0) \chi''  (y/r_0) \Bigr).
$$
The only values of $z_0$ for which $F(x/r_0) \chi''  (y/r_0) \neq
0$ or $\chi  (y/r_0) <1$
 verify $|z_0|>r_0$, and at those points
we have, using the fact that $F''(x/r_0) <0$,
$$
\frac12 \Delta \psi (|z_0|) + \Delta \Phi (z_0) \le -4 +
\frac1{r_0^2} C_1 \|\chi''\|_\infty \le -1
$$
if we choose $C_1$ small enough. Hence we have strict
pseudoconvexity again.

So we may restrict attention to $|y|\le r_0$ and $\Delta \Phi (z_0)
=  F''(x/r_0) /r_0^2.$ Therefore
$$
\frac12 \Delta \psi (|z_0|) + \Delta \Phi (z_0) \le  2 - C_2/r_0^2
< - 2
$$
for a $C_2$ chosen large enough.

Finally, notice that points $(z_0,w_0)$ with $|z_0|<r_0/2$ and
$F(x)=0$ verify $(z_0,w_0) \in \partial D_0 \cap
\partial D$,  $ D_0 \subset D$, and $D_0$ is
strictly pseudoconcave at $(z_0,w_0)$, so $D$ is as well.
\end{proof}

\noindent{\it Proof of Proposition \ref{1}.} Let $D$ be the domain
from Lemma \ref{6}. We may choose a dense countable subset
$(a_j)\subset\partial D$ of points of strict pseudoconvexity. For
any $j,$ there is a negative function $u_j\in\PSH(D)$ with
$\lim_{z\to a_j}u_j(z)=0.$ If $(D_j)$ is an exhaustion of $D$ such that $D_j\Subset
D_{j+1}$ and
$m_j=-\sup_{D_j}{u_j},$ then it is enough to take $u$ to be the upper
semicontinuous regularization of $\sup_j u_j/m_j.$\qed

\section{Proofs of Propositions 2, 3 and 4}

\noindent{\it Proof of Proposition \ref{2}.} We may assume that
$D$ has a global defining function $r:U\to\RR$ with $U=U(\partial
D)$, $r\in\mathcal{C}^2(U)$, and $\grad r\neq 0$ on $U$, such that
$D\cap U=\{z\in U:r(z)<0\}$.

Now assume the contrary. Then we may find a point $z^0\in\partial
D$ such that the Levi form of $r$ at $z^0$ is not positive
semidefinite on the complex tangent hyperplane to $\partial D$ at
$z_0.$ Therefore, there is a complex tangent vector $a$ with $\L
r(z_0,a)\le-2c<0$, where $\L r(z_0,a)$ denotes its Levi form at
$z^0$ in direction of $a$. Moreover, we may assume that
$|\frac{\partial r}{\partial z_1}(z_0)|\geq 2c.$

Now choose $V=V(z^0)\subset U$ and $u\in\PSH(D\cap V)$ with
$$\limsup_{D\cap V\owns z\to z_0}u(z)=\infty; $$ in particular,
there is a sequence of points $D\cap V\ni b^j\to z_0$ such that
$u(b^j)\to\infty$.

By the $\mathcal{C}^2$-smooth assumption, there is an $\eps_0>0$
such that for all $z\in\BB(z_0,\eps_0)\subset V$ and all $\tilde
a\in\BB(a,\eps_0)$ we have

$$
\L r(z,\tilde a)\leq -c,\quad |\frac{\partial r}{\partial
z_1}(z)|\geq c.
$$

Now fix an arbitrary boundary point $z\in\partial
D\cap\BB(z_0,\eps_0)$. Define
$$
a(z):=a+(-\frac{\sum_{j=1}^n a_j\frac{\partial r}{\partial
z_j}(z)}{\frac{\partial r}{\partial z_1}(z)},0,\dots,0).
$$
Observe that this vector is a complex tangent vector at $z$ and
$a(z)\in\BB(a,\eps_0)$ if $z\in\BB(z_0,\eps_1)$ for a sufficiently
small $\eps_1<\eps_0$.

Now, let $z\in\partial D\cap\BB(z_0,\eps_1)$. Put
$$
b_1(z):=\frac{\L r(z,a(z))}{2\frac{\partial r}{\partial z_1}(z)}
$$
and
$$
\phi_z(\lambda)=z+\lambda a+(\lambda a_1(z)+\lambda^2
b_1(z),0,\dots,0),\quad \lambda\in\CC.
$$

Moreover, if $\eps_1$ is sufficiently small, we may find $\delta,
t_0>0$ such that for all $z\in\partial D\cap\BB(z_0,\eps_1)$ we
have
$$
\overline D\cap \BB(z,\delta)-t\nu(z)\subset D,\quad 0<t\leq t_0,
$$
where $\nu(z)$ denotes the outer unit normal vector of $D$ at $z$.

Next using the Taylor expansion of $\phi_z$, $z\in\partial
D\cap\BB(z_0,\eps_1)$, $\eps_1$ sufficiently small, we get
$$
r\circ\phi_z(\lambda)=|\lambda|^2\Bigl(\L
r(z,a(z))+\eps(z,\lambda\Bigr),
$$
where $|\eps(z,\lambda)|\leq \eps(\lambda)\to 0$ if $\lambda\to
0$.

In particular, $\phi_z(\lambda)\in \BB(z,\delta)\cap D\subset
V\cap D$ when $0<|\lambda|\leq \delta_0$ for a certain positive
$\delta_0$ and $r\circ\phi_z(\lambda)\leq -\delta_0^2c/2$ when
$|\lambda|=\delta_0$.

Hence, $K:=\bigcup_{z\in\partial D\cap \BB(z_0,\eps_1),
|\lambda|=\delta_0}\phi_z(\lambda)\Subset D\cup V$. Choose an open
set $W=W(K)\Subset D\cap V$. Then $u\leq M$ on $W$ for a positive
$M$.

Finally, choose a $j_0$ such that $b^j=z^j-t_j\nu(z^j)$, $j\geq
j_0$, where $z^j\in\partial D\cap\BB(z_0,\eps_1)$, $0<t_j\leq
t_0$, and $\phi_{z^j}(\lambda)\in W$ when $|\lambda|=\delta_0$.
Therefore, by construction, $u(b^j)\leq M$, which contradicts the
assumption.\qed
\smallskip

\noindent{\it Proof of Proposition \ref{3}.} Assume that $D$ is
not pseudoconvex. Then, by Corollary 4.1.26 in \cite{Hor}, there
is $\varphi\in\mathcal O(\Bbb D, D)$ such
$\dist(\varphi(0),\partial D)<\dist(\varphi(\zeta),\partial D)$
for any $\zeta\in\Bbb D_\ast.$ To get a contradiction, it remains
to use similar arguments as in the previous proof and we skip the
details.
\smallskip

\noindent{\it Proof of Proposition \ref{4}.} It is enough to show
that if $\mathcal O(\Bbb D,D)\ni\psi_j\to\psi$ and
$\psi(\zeta)\in\partial D$ for some $\zeta\in\Bbb D,$ then
$\psi(\Bbb D)\subset\partial D.$ Suppose the contrary. Then it is
easy to find points $\eta_k\to\eta\in\Bbb D$ such that
$\psi(\eta_k)\in D$ but $a=\psi(\eta)\in\partial D.$ We may assume
that $\eta=0$ and $g_a=\frac{1}{f_a}$ is bounded on $D\cap U_a.$ Let
$r\in(0,1)$ be such that $\psi(r\Bbb D)\Subset U_a.$ Then
$\psi_j(r\Bbb D)\subset U_a$ for any $j\ge j_0.$ Hence
$|g_a\circ\psi_j|<1$ and we may assume that $g_a\circ\psi_j\to
h_a\in\mathcal O(r\Bbb D,\Bbb C).$ Since $h_a(\eta)=0,$ it follows
by the Hurwitz theorem that $h_a=0.$ This contradicts the fact that
$h_a(\eta_k)=g_a\circ\psi(\eta_k)\neq 0$ for $|\eta_k|<r.$ \qed

\section{Proofs of Propositions 5, 6 and 7}

\noindent{\it Proof of Proposition \ref{10}.} Our aim is to
construct a non-pseudoconvex bounded domain $D\subset\CC^2$ such
that $\limsup_{z\to a}K_D(z)=\infty$ for any $a\in\partial D$.

Let us start with the domain $P\times\Bbb D$, where
$P=\{\lambda\in\Bbb C:\frac{1}{2}<|\lambda|<\frac{3}{2}\}$. Let
$$
S:=\{(z_1,z_2)=(x_1+iy_1,z_2)\in
P\times\DD:(x_1-1)^2+\frac{1+|z_2|^2}{1-|z_2|^2}y_1^2=\frac{1}{4},y_1>0\}.
$$
Define $D:=(P\times\DD)\setminus S$. Note that $D$ is a domain.
Its envelope of holomorphy is non-schlicht and consists of the
union of $D$ and one additional 'copy' of the set
$$
D_1:=\{(z_1,z_2)\in
P\times\DD:(x_1-1)^2+\frac{1+|z_2|^2}{1-|z_2|^2}y_1^2\leq\frac{1}{4},y_1>0\}.
$$
In particular, $D$ is not pseudoconvex. Note that convexity of the
the interior $D^0$ of $D_1$ implies that $\lim_{z\to\partial
D_1}K_{D^0}(z)=\infty$. Therefore, it follows from the localization
result for the Bergman kernel due to Diederich-Fornaess-Herbort
formulated for Riemann domains in the paper \cite{Irg1} that for all
$a\in S\subset\partial D_1$ the following property holds:
$\lim_{D\cap D_1\owns z \to a}K_D(z)=\infty$ (on the other hand
while tending to the points from $S$ from the 'other side' of the
domain $D$ the Bergman kernel is bounded from above). Obviously
$P\times\DD$ is Bergman exhaustive, so for any
$a\in\partial(P\times\DD)$ the following equality holds $\lim_{z\to
a}K_D(z)=\infty$.\qed
\smallskip

\noindent{\it Proof of Proposition \ref{11}.} Recall the following
facts that follow from \cite{Bre}.

If the envelope of holomorphy $\hat D$ of the domain $D$ is a domain
in $\Bbb C^n$ (is schlicht) then the Bergman kernel $K_D$ extends to
a real analytic function $\tilde K_D$ defined on $\hat D$.

Let $\emptyset\neq P_0\subset D$, $P_0\subset P$, $P\setminus
D\neq\emptyset$ and $\bar P_0\cap (\Bbb C^n\setminus
D)\neq\emptyset$, where $P_0,P$ are polydiscs, and the following
property is satisfied: for any $f\in\O(D)$ there is a function
$\tilde f\in\O(P)$ such that $f=\tilde f$ on $P_0$. Then the
Bergman kernel $K_D$ extends to a real analytic function on $P$.
More precisely, there is a real analytic function $\tilde K_D$
defined on $P$ such that $\tilde K_D(z)=K_D(z)$, $z\in P_0$.

Both facts above complete the proof of Proposition \ref{11}.\qed
\smallskip

The proof of Proposition \ref{12} is essentially contained in
\cite{Irg2}. However, this PhD Thesis is not publically accessible.
Therefore we repeat it here. The idea is the following: if
$\limsup_{z\to a}K_D(z)=\infty$ for some $a\in\partial D,$ then
there is an $f\in L_h^2(D)$ such that $\limsup_{z\to
a}|f(z)|=\infty$.
\smallskip

\noindent{\it Proof of Proposition \ref{12}.} In view of Proposition
5, $\limsup_{z\to a}K_D(z)=\infty$ for any $a\in\partial D$.

Let $a\in\partial D$. We claim that there is an $L_h^2(D)$-function
$h$ which is unbounded near $a$.

Assume the contrary. Hence for any $f\in L_h^2(D)$ there exists a
neighborhood $U_f$ of $a$ and a number $M_f$ such that $|f|\leq M_f$
on $D\cap U_f.$

Denote by $L$ the unit ball in $L_h^2(D)$ and by $c=\pi^n$.

Let $K_1:=\{z\in D:\dist(z,\partial D)\geq 1\}$ (if this is empty
take a smaller number than $1$). By the meanvalue inequality we have
for any $f\in L$ that $|f|\leq c$ on $K_1$. By assumption, there are
$z_1\in D$ and $f_1\in L$ such that $|z_1-a|<1$ and $|f_1(z_1)|>c$.

Set $g_1:=f_1/c$. Then $g\in L$ and therefore there are a
neighborhood $U_1$ of $a$ and number $M_1>1$ such that $|g_1|\leq
M_1$ on $D\cap U_1$.

Set $K_2:=\{z\in D:\dist(z,\partial D)\geq\dist(z_1,\partial D)\}$
and $d=c\dist(z_1,\partial D).$ Then $K_1\subset K_2$. Choose
$z_2\in U_1\cap D$, $z_2\notin K_2,$ $|z_2-a|<1/2,$ and $f_2\in L$
with $|f_2(z_2)|\geq d(1^3+1^2M_1)$. Moreover, $|f_2|\leq d$ on
$K_2$. Put $g_2:=f_2/d$. Then $g_2\in L$. Choose now a neighborhood
$U_2$ of $a$ and a number $M_2$ such that $|g_2|\leq M_2$ on $D\cap
U_2$.

Then we continue this process.

So we have points $z_k\in K_{k-1}$, $z_k\notin K_{k-1}$,
$|z_k-a|<1/k,$ and functions $f_k\in L$ with
$$|f_k(z_k)|\geq c\dist(z_{k-1},\partial
D)^n(k^3+k^2\sum_{j=1}^{k-1}M_j).$$

Setting $g_k:=f_k/d$ and $h:=\sum_{j=1}^\infty g_j/j^2,$ it is
clear that $h\in L_h^2(D)$.

Fix now $k\ge 2$. Then $$|h(z_k)|\geq
\frac{|g_k(z_k)|}{k^2}-\sum_{j=1}^{k-1}\frac{|g_j(z)|}{j^2}-
\sum_{j=k+1}^\infty\frac{|g_j(z)|}{j^2}$$
$$\ge
k+\sum_{j=1}^{k-1}M_j-\sum_{j=1}^{k-1}\frac{M_j}{j^2}-\sum_{j=k+1}^\infty
\frac{1}{j^2}>k-\frac{1}{6}.$$ In particular, $h$ is unbounded at
$a$ which is a contradiction.

It remains to choose a dense countable sequence
$(a_j)\subset\partial D$ such that any term repeats infinitely
many times and to copy the proof of the Cartan-Thullen theorem.
\qed

\end{document}